\documentclass[11pt]{amsart}
\usepackage{a4wide,amsmath,amssymb}
\usepackage[toc,page]{appendix}
\usepackage{bbm}
\usepackage{hyperref}
\usepackage{url}
\usepackage[normalem]{ulem}
\usepackage{courier}
\usepackage{easyReview}
\usepackage{mathtools}
\usepackage{xcolor}
\usepackage{graphicx}

\newtheorem{thm}{Theorem}[section]
\newtheorem{lemma}[thm]{Lemma}
\newtheorem{prop}[thm]{Proposition}
\newtheorem{cor}[thm]{Corollary}
\theoremstyle{remark}
\newtheorem{rem}[thm]{Remark}

\theoremstyle{definition}
\newtheorem{defn}[thm]{Definition}
\newtheoremstyle{Claim}{}{}{\itshape}{}{\itshape\bfseries}{:}{ }{#1}
\theoremstyle{Claim}

\newcommand{\T}{{\mathbb{T}}}

\newcommand{\R}{\mathbb{R}}

\newcommand{\mI}{\mathcal I}

\newcommand{\mA}{\mathcal A}

\newcommand{\virg}{``}

\newcommand{\half}{\frac 1 2}

\theoremstyle{plain}

\def\sideremark#1{\ifvmode\leavevmode\fi\vadjust{
		\vbox to0pt{\hbox to 0pt{\hskip\hsize\hskip1em
				\vbox{\hsize3cm\tiny\raggedright\pretolerance10000
					\noindent #1\hfill}\hss}\vbox to8pt{\vfil}\vss}}}

\newcommand{\dt}{{\Delta t}}

\makeatletter
\@namedef{subjclassname@2020}{\textup{2020} Mathematics Subject Classification}
\makeatother

\begin{document}
	
	\title[]{$L^p$ Error Estimates for Numerical Approximation of Convex Hamilton--Jacobi Equations}
	
	\author{Alessio Basti}
	\author{Fabio Camilli}
\address{Engineering and Geology Department, University \virg G. d'Annunzio'' Chieti-Pescara,
	viale Pindaro 42, 65127 Pescara (Italy)}
\email{alessio.basti@unich.it}
\email{fabio.camilli@unich.it}


	\date{\today}
	
\begin{abstract}
	We establish $L^p$ error estimates for monotone numerical schemes approximating convex Hamilton--Jacobi equations on the $d$-dimensional torus. Using the adjoint method and semiconcavity estimates, we first prove an $L^1$ error bound of order one for classical monotone schemes of Crandall–Lions type and semi-Lagrangian schemes under standard convexity assumptions on the Hamiltonian and semiconcavity assumptions on the initial datum. By interpolation with the classical $L^\infty$ estimate, we obtain $L^p$ estimates for every $1\le p<+\infty$. 
\end{abstract}
	
	\maketitle
	\numberwithin{equation}{section}
	
\section{Introduction}
The Hamilton--Jacobi (HJ) equation
\begin{equation}\label{hj_intro}
	\begin{cases}
		\partial_t u + H(Du) = 0 & \text{in } Q := \mathbb{T}^d\times(0,T),\\[4pt]
		u(x,0) = g(x) & \text{on } \mathbb{T}^d,
	\end{cases}
\end{equation}
plays a central role in several areas of analysis and applied mathematics, including optimal control, front propagation, differential games, and mean field game theory \cite{achdou}. The viscosity solution theory provides a robust framework to study this problem: by exploiting the maximum principle, one can establish general results of existence, uniqueness, and regularity for viscosity solutions of \eqref{hj_intro}; see, for instance, \cite{bcd}.

From the numerical point of view, monotone approximation schemes such as finite-difference and semi-Lagrangian methods have been extensively investigated since the classical works of Crandall and Lions \cite{cl} and Barles and Souganidis \cite{bs}, where convergence to the viscosity solution was proved under suitable assumptions; see also \cite{ff} for a general review. For such schemes, \(L^\infty\) error bounds of order \(1/2\) were obtained in \cite{cl,cdi} for a broad class of monotone methods, without assuming convexity of the Hamiltonian. These estimates rely on classical arguments from viscosity solution theory, in particular the doubling-of-variables technique combined with the maximum principle.

For convex Hamiltonians, sharper information is available in weaker norms. Duality and adjoint arguments for first-order Hamilton--Jacobi equations go back to Kruzhkov \cite{kr} and were developed, in the context of \(L^1\)-stability and error estimates for approximate solutions, by Lin and Tadmor \cite{lt}. Their theory shows that semiconcave approximate solutions satisfy an \(L^1\) stability estimate controlled by the initial perturbation and by the \(L^1\)-size of the truncation error. They apply this principle to vanishing viscosity approximations and to Godunov-type finite-difference approximations constructed through an exact evolution step followed by a projection step. Related high-resolution nonoscillatory central schemes for Hamilton--Jacobi equations are studied in \cite{lt2}. We also mention the related use of weak distances in the numerical analysis of scalar conservation laws, for instance the \(W_1\)-convergence result of Fjordholm and Solem \cite{fs}.

In a systematic nonlinear form, the adjoint method was developed by L.C. Evans \cite{e}, and has proved to be a powerful tool to capture fine properties of solutions of Hamilton--Jacobi equations that are not accessible through standard viscosity theory. The key idea is to study the linearized adjoint equation associated with \eqref{hj_intro}, namely
\begin{equation}\label{adjoint_intro}
	\begin{cases}
		\partial_{t} \rho  + \operatorname{div} \big(D_pH(Du)\,\rho\big) = 0 & \text{ in } Q,\\[4pt]
		\rho(x,T) = \rho_T(x) & \text{ on } \mathbb{T}^d.
	\end{cases}
\end{equation}
By coupling the original nonlinear equation with its linearized adjoint, one can derive quantitative estimates for various asymptotic problems. This method has been successfully applied to obtain rates of convergence in the vanishing-viscosity approximation of Hamilton--Jacobi equations, both in the \(L^\infty\) norm \cite{e,t,cg,cd,go} and in \(L^p\) norms \cite{cgm}.

The purpose of this paper is to establish \(L^p\) error estimates for monotone schemes approximating convex Hamilton--Jacobi equations on the \(d\)-dimensional torus \(\mathbb{T}^d\), for both finite-difference and semi-Lagrangian discretizations. More precisely, we consider \eqref{hj_intro} under the assumptions that the Hamiltonian \(H\) is smooth, convex, and coercive, and that the initial datum \(g\) is Lipschitz continuous and semiconcave. For both classes of schemes, we first prove an \(L^1\) error estimate of order one. Interpolation between this \(L^1\) estimate and the classical \(L^\infty\) estimate of order \(1/2\) then yields \(L^p\) error bounds for every $p\in[1,+\infty]$.

The \(L^1\)-stability mechanism used here is the same semiconcavity-duality mechanism underlying Lin--Tadmor theory \cite{lt}. The difference is in the way the hypotheses of that mechanism are verified for the numerical schemes under consideration. Lin--Tadmor treat Godunov-type approximations based on a transport--projection structure, where the truncation error is generated by the projection step. Here, instead, we consider local monotone schemes of Crandall--Lions type, written directly in terms of a consistent, convex, monotone numerical Hamiltonian. These schemes are not, in general, transport--projection schemes in the sense of \cite{lt}, although special choices are closely related to the Godunov and Lax--Friedrichs literature. Accordingly, our finite-difference analysis proceeds in a direct way. We first work with a continuous-space, discrete-time auxiliary scheme and show that the monotone update preserves semiconcavity in the continuous sense. Thus no discrete-semiconcavity reconstruction is exploited. The local residual is then estimated by separating the contribution due to the time interpolation from the spatial finite-difference consistency error. This gives an \(L^1\) estimate of order $1$, and the corresponding fully discrete estimate is obtained by interpolation on the grid.

Compared with the work of Cagnetti--Gomes--Tran \cite{cgt}, where an adjoint-based strategy was applied to a continuous-time, discrete-space approximation in one space dimension for the Lax--Friedrichs numerical Hamiltonian, our finite-difference result is time-discrete, multidimensional, and applies to a more general class of convex monotone numerical Hamiltonians.

We also adapt the same general strategy to semi-Lagrangian schemes for Hamiltonians of the form \(H(x,p)=H_0(p)-V(x)\), with \(H_0\) convex. In that case the proof uses the dynamic-programming representation of the discrete value functions together with their Lipschitz and semiconcavity properties. This yields an \(L^1\) estimate of order \(1\) for the time-discrete semi-Lagrangian scheme, and the corresponding estimates for its fully discrete version. Interpolation with the \(L^\infty\) bounds gives the associated \(L^p\) convergence rates.

The remainder of the paper is organized as follows. In Section~\ref{sec:finite_difference} we analyze the finite-difference scheme: we prove the key \(L^1\) estimate and then obtain \(L^p\) bounds by interpolation. Section~\ref{sec:semi_lagrangian} is devoted to the semi-Lagrangian scheme, where we adapt the strategy to the discrete dynamic-programming formulation. Some technical proofs and auxiliary lemmas are collected in the Appendix.

\subsection{Assumptions and preliminaries}
We work on the $d$-dimensional flat torus $\T^d$, which we identify with $[0,1]^d$ with periodic boundary conditions. 
In the following the notion of semiconcavity plays a crucial role. For a concise account of the theory and its applications to Hamilton-Jacobi equations we refer the reader to \cite{cs}. We recall the definition and two properties which we repeatedly use.

\begin{defn}
	A function \(u:\mathbb{T}^d\to\mathbb{R}\) is \emph{semiconcave} with constant \(C_{conc}>0\) if for every \(x_1,x_2\in\mathbb{T}^d\) and every \(\lambda\in(0,1)\)
	\[
	u(\lambda x_1+(1-\lambda)x_2)\ge \lambda u(x_1)+(1-\lambda)u(x_2)-\lambda(1-\lambda)C_{conc}|x_1-x_2|^2.
	\]
\end{defn}

\begin{lemma}\label{equivdefsemicon} 
	For a function $u:\T^d\to \R$, the following assertions are equivalent:
	\begin{itemize}
		\item[(i)] The function $u$ is semiconcave, with constant $C_{conc}$.
		\item[(ii)] For all $x,y\in \T^d$, we have
		$$ u(x+y)+ u(x-y) -2u(x) \leq C_{conc} |y|^{2}.$$
		\item[(iii)] Setting $I_{d}$ for the identity matrix, we have  that  $D^{2}u \leq C_{conc} I_{d}$ in the sense of distributions where $D^2 u$ denotes the classical Hessian defined a.e.
	\end{itemize}
\end{lemma}

For the following results, we refer to \cite[Thm. 2.1.7, Thm. 2.3.1, and Prop. 1.1.3 ]{cs}, and \cite{kr2}.

\begin{lemma}\label{lemma:bound_hessian} 
	Let $u:\T^d \to \R$ be semiconcave. Then: 
	\begin{itemize}
		\item[(i)] $u$ is  Lipschitz continuous.
		\item[(ii)] $Du\in BV(\T^d;\R^d)$ and
		\[
		\int_{\T^d}|D^2 u(x)|\,dx \le C_{\mathrm{Hess}},
		\]
		where $C_{\mathrm{Hess}}$ depends only on the semiconcavity constant of $u$. With a slight abuse of notation, $\int_{\T^d}|\cdot|\,dx$ also denotes the total variation norm.
	\end{itemize}
\end{lemma}

Throughout the paper, we make the following standing assumptions:
\begin{itemize}
	\item[(H1)] $H\in C^2(\mathbb{R}^d)$ is convex.
	\item[(H2)] $H$ is coercive: \(\displaystyle\lim_{|p|\to\infty}H(p)=+\infty\).
	\item[(G1)] The initial datum \(g:\mathbb{T}^d\to\mathbb{R}\) is Lipschitz continuous and semiconcave.
\end{itemize}

\section{Estimates for the finite difference scheme}\label{sec:finite_difference}
In this section, we first derive an error estimate for a semi-discrete finite difference approximation. Building on this result, we then obtain an error estimate for the fully discrete finite difference scheme.\\
From now on we fix a uniform space step \(h>0\) and a time step \(\Delta t>0\) such that \(N\Delta t=T\), and denote \(\varepsilon=(h,\Delta t)\). For each coordinate direction $i\in\{1,\dots,d\}$, we define the forward and backward difference operators
\[
\delta_{h}^{(i)} v(x) = \frac{v(x+h e_i)-v(x)}{h}\qquad
\delta_{-h}^{(i)} v(x) = \frac{v(x)-v(x-h e_i)}{h},
\]
and
\begin{align*}
\delta_{h}  v(x)=( \delta_{h}^{(1)} v(x), \dots, \delta_{h}^{(d)} v(x) ),\qquad
\delta_{-h}  v(x)=( \delta_{-h}^{(1)} v(x), \dots, \delta_{-h}^{(d)} v(x) ).
\end{align*}
We consider the explicit semi-discrete finite-difference scheme 
\begin{equation}\label{HJ_app_d}
\begin{cases}
u^\varepsilon_{n+1}(x)-u^\varepsilon_n(x) + \Delta t\, F\big(-\delta_{h} u^\varepsilon_n(x),\ \delta_{-h} u^\varepsilon_n(x)\big)=0, & x\in\mathbb{T}^d,\ n=0,\dots,N-1,\\[4pt]
u^\varepsilon_0(x)=g(x), & x\in\mathbb{T}^d,
\end{cases}
\end{equation}

The numerical Hamiltonian is a function \(F:\mathbb{R}^d\times\mathbb{R}^d\to\mathbb{R}\) satisfying:

\begin{itemize}
    \item[(F1)] \(F\in C^2(\mathbb{R}^{2d})\) is convex.  
    \item[(F2)] (Consistency) For every \(p\in\mathbb{R}^d\),
       $ F(-p,p)=H(p).$
    \item[(F3)] (Monotonicity) Let $R>0$ be a constant. We assume $\varepsilon=(h,\Delta t)$ is chosen such that the update operator
    \[
       G\big(v(x), \{v(x \pm h e_i)\}_{i=1}^d\big):= v(x) - \Delta t\, F\big(-\delta_{h} v(x),\ \delta_{-h} v(x)\big)
    \]
    is a nondecreasing function of each of the grid values involved (i.e., the central $v(x)$ and the neighboring $v(x \pm h e_i)$), provided that the absolute value of the discrete slopes $-\delta_{h} v(x)$ and $\delta_{-h} v(x)$ is bounded by $R$.
\end{itemize}
\begin{rem}
    The monotonicity assumption in (F3) encapsulates two distinct requirements:
    \begin{enumerate}
        \item Requiring $G$ to be nondecreasing with respect to the $v(x\pm h e_i)$ implies that $F$ must be componentwise nondecreasing in its arguments, i.e. $F_{p_i} \ge 0$ and $F_{q_i} \ge 0$.
        \item Requiring $G$ to be nondecreasing with respect to the central value $v(x)$ leads to the CFL condition:
        \[
            \Delta t \le \frac{h}{M_F(R)}, \quad \text{where } M_F(R) := \sup_{|p|,|q|\le R} \sum_{i=1}^d \big( F_{p_i}(p,q)+F_{q_i}(p,q) \big).
        \]
    \end{enumerate}
\end{rem}
Two representative numerical Hamiltonians satisfying (F1)--(F3) are the following (see \cite{cl} for additional examples). The first is a multidimensional Lax-Friedrichs type
\[
F(p,q)=H\!\Big(\frac{q-p}{2}\Big)+\alpha\sum_{i=1}^d (p_i+q_i),
\qquad \alpha>0,
\]
where the linear dissipation (taken with \(\alpha\) sufficiently large) is used to enforce componentwise monotonicity. The second example is the coordinate-wise piecewise separable form; we state it for \(d=1\) (see \cite{cgt}): assume \(H(0)=0=\min_{p\in\mathbb{R}}H(p)\) and set
\[
F(p,q)=F_1(p)+F_2(q),\qquad
F_1(p)=\begin{cases}0,&p\le0,\\[2pt]H(-p),&p>0,\end{cases}\qquad
F_2(q)=\begin{cases}0,&q\le0,\\[2pt]H(q),&q>0.\end{cases}
\]
On each half-line \(F_1,F_2\) inherit the regularity of \(H\), and this choice yields a monotone, separable numerical Hamiltonian.

We show some important properties of the solution of the problem \eqref{HJ_app_d}
(for the proof, see the Appendix)
\begin{prop}[Stability and semiconcavity]\label{prop:properties_scheme}
Let \(\{u^\varepsilon_n\}_{n=0}^N\) be the solution of \eqref{HJ_app_d}. 
Under assumptions \textup{(H1)--(H2)}, \textup{(G1)} (with semiconcavity constant \(C_{conc}\)), and \textup{(F1)--(F3)}, the following properties hold for every \(n=0,\dots,N\):
\begin{enumerate}
  \item[\textup{(i)}] Lipschitz bounds (discrete and continuous).
  \[
    \|\delta_h u^\varepsilon_n\|_{L^\infty(\T^d)} \le \|\delta_h g\|_{L^\infty(\T^d)},
    \qquad
    \|Du^\varepsilon_n\|_{L^\infty(\T^d)} \le \|Dg\|_{L^\infty(\T^d)},
  \]
  where the gradient \(Du^\varepsilon_n\) is a.e.
  \item[\textup{(ii)}] Stability.
  \[
    \|u^\varepsilon_n\|_{L^\infty(\T^d)} \le \|g\|_{L^\infty(\T^d)} + n\Delta t\,|H(0)|.
  \]
  \item[\textup{(iii)}] Semiconcavity. For every \(x\in\T^d\) and every \(k\in\T^d\setminus\{0\}\) define
  \[
    \Delta_k u^\varepsilon_n(x):=u^\varepsilon_n(x+k)+u^\varepsilon_n(x-k)-2u^\varepsilon_n(x).
  \]
  Then
  \begin{equation}\label{semiconc_solutions_finite_diff}
    \sup_{x\in\T^d,\; k\in \T^d\setminus\{0\}}\frac{\Delta_k u^\varepsilon_n(x)}{|k|^2} \le C_{conc},
  \end{equation}
  i.e. the semiconcavity constant of \(u^\varepsilon_n\) is bounded by that of the initial datum \(g\).
\end{enumerate}
\end{prop}

By a classical stability result in viscosity solution theory, the solution of the scheme \eqref{HJ_app_d} converges to the viscosity solution of \eqref{hj_intro} (cf. \cite{bs}). Furthermore, an estimate of the approximation error is provided in \cite{cl}.
\begin{thm}[Convergence rate in \(L^\infty\)]\label{thm:Linfty_estimate}
Let \(\varepsilon=(h,\Delta t)\) and let \(\{u^\varepsilon_n\}_{n=0}^N\) be the solution of \eqref{HJ_app_d}. Let \(u\) denote the viscosity solution of \eqref{hj_intro}. Then there exists a constant \(C>0\), independent of \(\varepsilon\), such that
\begin{equation}\label{stima_Linfty_d}
  \sup_{n=0,\dots,N}\big\|u^\varepsilon_n(\cdot) - u(\cdot,n\Delta t)\big\|_{L^\infty(\T^d)}
  \le C\,|\varepsilon|^{1/2}.
\end{equation}
where $|\varepsilon|=h+\Delta t$.
\end{thm}
\begin{rem}
This result holds under fairly general assumptions and does not require   convexity assumptions. In particular, it applies to the setting considered here.
\end{rem}
We now prove a $L^1$-estimate for the scheme \eqref{HJ_app_d}. We need a preliminary property (for the proof see the Appendix).
\begin{lemma}[Finite-difference properties]\label{lemma:finite_difference}
	Let $v:\T^d\to\R$ be semiconcave with constant $C_{conc}$.
	Then there exists a constant $C_1$ depending only on $C_{conc}$ such that, for every \(h>0\),
	\begin{equation}\label{eq:finite_diff_est_d}
		\int_{\T^d} \big|\delta_{h}^{(i)} v(x) - D_i v(x)\big|\,dx
		\;\le\; C_1\, h, \qquad
		\int_{\T^d} \big|\delta_{-h}^{(i)} v(x) - D_i v(x)\big|\,dx
		\;\le\; C_1\, h,
	\end{equation}
	for each \(i=1,\dots,d\).
\end{lemma}
\begin{thm}[$L^1$ error estimate; finite difference scheme]\label{thm:stima_L1_finitediff}
Given $\varepsilon=(h,\Delta t)$, let $u$ be the viscosity solution of \eqref{hj_intro} and \(\{u^\varepsilon_n\}_{n=0}^N\) the solution of \eqref{HJ_app_d}. Under (H1)--(H2), (G1), and (F1)--(F3), there exists \(C>0\) independent on \(\varepsilon\), such that
\begin{equation}\label{stima_L1}
  \sup_{n=0,\dots,N}\big\|u_n^\varepsilon(\cdot)-u(\cdot,n\Delta t)\big\|_{L^1(\mathbb{T}^d)}
  \le C\,|\varepsilon| \;=\; C\,(h+\Delta t).
\end{equation}

\end{thm}
\begin{proof}
Define the time-linear interpolation of the solution of the scheme \eqref{HJ_app_d} as
\begin{equation}\label{time_interpolation}
	u^\varepsilon(x,t)=\Big(1-\frac{t-n\Delta t}{\Delta t}\Big)u^\varepsilon_n(x)+\frac{t-n\Delta t}{\Delta t}u^\varepsilon_{n+1}(x),
\end{equation}
for \(t\in[n\Delta t,(n+1)\Delta t]\), with \(n\in\{0,\dots,N-1\}\) .
 First note that \(\partial_t u^\varepsilon=(u^\varepsilon_{n+1}-u^\varepsilon_n)/\Delta t\). Using the compatibility condition \(F(-p,p)=H(p)\) we obtain the identity
\[
\partial_t u^\varepsilon + H(D u^\varepsilon)
= F(-D u^\varepsilon, D u^\varepsilon) - F(-\delta_h u^\varepsilon_n,\delta_{-h} u^\varepsilon_n),
\qquad t\in[n\Delta t,(n+1)\Delta t].
\]
Set \(z:=u^\varepsilon-u\). Subtracting the PDE for \(u\) from the above yields, for \(t\in[n\Delta t,(n+1)\Delta t]\),
\begin{equation}\label{HJ_pert_d_proof}
  \partial_t z + b(x,t)\cdot D z
  \;=\;
  F(-D u^\varepsilon, D u^\varepsilon) - F(-\delta_h u^\varepsilon_n,\delta_{-h} u^\varepsilon_n),
\end{equation}
where
\[
b(x,t):=\int_0^1 D_pH\big(s\,D u^\varepsilon+(1-s)\,D u\big)\,ds \in\R^d.
\]
Fix \(\tau=\bar n\,\Delta t\) with \(\bar n\in\{1,\dots,N\}\). Consider the adjoint (backward transport) equation
\begin{equation}\label{adjoint_d}
  -\partial_t \rho - \operatorname{div}(b \rho)=0\quad\text{in }\T^d\times(0,\tau),\qquad
  \rho(\cdot,\tau)=\operatorname{sgn}(z(\cdot,\tau)).
\end{equation}
Well posedness of \eqref{adjoint_d} and uniform boundness of the related solution are discussed in \cite[pag. 710]{lt} and \cite[Thm. 4.12]{f}. Note that we have \(\|\rho(\cdot,\tau)\|_{L^\infty(\T^d)}\le 1\).\\
Multiply \eqref{HJ_pert_d_proof} by \(\rho\) and integrate over \(\T^d\times(0,\tau)\).  Using integration by parts, the adjoint equation \eqref{adjoint_d} and  $z(\cdot,0)=0$, the left-hand side term gives
\[
\int_0^\tau\!\!\int_{\T^d} \big(\partial_t z+ b\cdot Dz\big)\rho\,dxdt
= \int_{\T^d} z(x,\tau)\rho(x,\tau)\,dx - \int_{\T^d} z(x,0)\rho(x,0)\,dx= \int_{\T^d} |z(\tau,x)|\,dx .
\]
Hence we get
\begin{equation}\label{duality_identity_d}
\int_{\T^d} |z(\tau,x)|\,dx
=\sum_{n=0}^{\bar n-1}\int_{n\Delta t}^{(n+1)\Delta t}\!\!\int_{\T^d}\big[F(-D u^\varepsilon, D u^\varepsilon)-F(-\delta_h u^\varepsilon_n,\delta_{-h} u^\varepsilon_n)\big]\rho\,dxdt.
\end{equation}
Accordingly write the right-hand side of \eqref{duality_identity_d} as \(I_1+I_2\), where
\[
\begin{aligned}
I_1 &:=\sum_{n=0}^{\bar n-1}\int_{n\Delta t}^{(n+1)\Delta t}\!\!\int_{\T^d}\big[F(-D u^\varepsilon, D u^\varepsilon)-F(-D u^\varepsilon_n, D u^\varepsilon_n)\big]\rho\,dxdt,\\
I_2 &:=\sum_{n=0}^{\bar n-1}\int_{n\Delta t}^{(n+1)\Delta t}\!\!\int_{\T^d}\big[F(-D u^\varepsilon_n, D u^\varepsilon_n)-F(-\delta_h u^\varepsilon_n,\delta_{-h} u^\varepsilon_n)\big]\rho\,dxdt.
\end{aligned}
\]

\emph{Estimate of \(I_1\).}
Using the compatibility \(F(-p,p)=H(p)\), and linearizing \(H\), we get, for each \(t\in[n\Delta t,(n+1)\Delta t]\),
\[
F(-D u^\varepsilon, D u^\varepsilon)-F(-D u^\varepsilon_n, D u^\varepsilon_n)
=H(D u^\varepsilon)-H(D u^\varepsilon_n)\]\[
=\Big(\int_0^1 D_pH(sD u^\varepsilon+(1-s)D u^\varepsilon_n)\,ds\Big)\cdot\big(D u^\varepsilon-D u^\varepsilon_n\big).
\]
Let \(M:=\sup_{|p|\le R}|D_pH(p)|\) with \(R:= \|Dg\|_{L^\infty(\T^d)}\). Then
\[
|I_1|\le M\|\rho\|_{L^\infty(\T^d\times(0,\tau))}\sum_{n=0}^{\bar n-1}\int_{n\Delta t}^{(n+1)\Delta t}\int_{\T^d}|D u^\varepsilon(t)-D u^\varepsilon_n|\,dxdt.
\]
We estimate \(D u^\varepsilon(t)-D u^\varepsilon_n\). For \(t\in[n\Delta t,(n+1)\Delta t]\) one has
by \eqref{HJ_app_d}
\[
\big|D u^\varepsilon(t)-D u^\varepsilon_n\big|
=\bigg|\frac{t-n\Delta t}{\Delta t}D\big(u^\varepsilon_{n+1}- u^\varepsilon_n\big)
\bigg|\leq\,\big|\Delta t\,DF(-\delta_h u^\varepsilon_n,\delta_{-h} u^\varepsilon_n)\big|.
\]
By the chain rule and the uniform bounds on the partial derivatives of $F$, we have pointwise
\[
\big|D F(-\delta_h u^\varepsilon_n,\delta_{-h} u^\varepsilon_n)(x)\big|
\le L\sum_{i,j=1}^d\Big(\big|D_{j}(\delta_h^{(i)}u^\varepsilon_n)(x)\big|
+\big|D_{j}(\delta_{-h}^{(i)}u^\varepsilon_n)(x)\big|\Big),
\]
where $L:=\max_{1\le i\le d}\ \sup_{|(p,q)|\le R}\Big\{\;|F_{p_i}(p,q)|,\;|F_{q_i}(p,q)|\;\Big\}<\infty.$ For fixed $i,j$ one has the representation
\[
D_j\big(\delta_h^{(i)} u^\varepsilon_n\big)(x)=\frac{1}{h}\bigg(D_j u^\varepsilon_n(x + h e_i) - D_j u^\varepsilon_n(x)\bigg)
= \frac{1}{h}\int_0^h D_{ij} u^\varepsilon_n(x + s e_i)\,ds
= \int_0^1 D_{ij} u^\varepsilon_n(x + \theta h e_i)\,d\theta,
\]
and similarly for the backward difference. Integrating the above pointwise bound in $x$ and changing variables yields
\[
\int_{\T^d}\big|D F(-\delta_h u^\varepsilon_n,\delta_{-h} u^\varepsilon_n)(x)\big|\,dx
\le 2L\sum_{i,j=1}^d\int_{\T^d}|D_{ij}u^\varepsilon_n(y)|\,dy
=2L\int_{\T^d}|D^2 u^\varepsilon_n(y)|\,dy.
\]
By Lemma \ref{lemma:bound_hessian} (applied to $u^\varepsilon_n$, using the semiconcavity preserved by the scheme, Proposition \ref{prop:properties_scheme}) we have a uniform bound by $C_{Hess}$. Plugging into the bound for \(I_1\) and integrating in time on \([n\Delta t,(n+1)\Delta t]\) yields
\[
|I_1|\le 2L\,C_{Hess}M\|\rho\|_{L^\infty(\T^d\times(0,\tau))}\,\sum_{n=0}^{\bar n-1}\int_{n\Delta t}^{(n+1)\Delta t}\Delta t\,dt
\leq C_{I_1}\Delta t.
\]
\emph{Estimate of \(I_2\).}
Linearizing \(F\), and using the previously defined constant $L$, we get 
\[
|F(-D u^\varepsilon_n, D u^\varepsilon_n)-F(-\delta_h u^\varepsilon_n,\delta_{-h} u^\varepsilon_n)|\leq L \sum_{i=1}^d(|\delta_h^{(i)}u^\varepsilon_n-D_{i}u^\varepsilon_n|+|D_{i}u^\varepsilon_n-\delta_{-h}^{(i)}u^\varepsilon_n|)
\]
Therefore, by relying on Lemma \ref{lemma:finite_difference}, 
\begin{equation}\label{I2_bound}
|I_2|
\le
L\|\rho\|_{L^\infty(\T^d\times(0,\tau))}
\sum_{n=0}^{\bar n-1}
\int_{n\Delta t}^{(n+1)\Delta t}
\int_{\T^d}
\sum_{i=1}^d
\Big(
|\delta_h^{(i)}u_n^\varepsilon-D_i u_n^\varepsilon|
+
|\delta_{-h}^{(i)}u_n^\varepsilon-D_i u_n^\varepsilon|
\Big)\,dx\,dt\leq C_{I_2} h
\end{equation}
Combining the estimates for \(I_1\) and \(I_2\), we have
\[
\|u^\varepsilon(\cdot,\tau)-u(\cdot,\tau)\|_{L^1(\T^d)}
\le \max\{C_{I_1},C_{I_2}\}(\Delta t + h).
\]
\end{proof}

\begin{cor}\label{cor:stima_Lp}
Assume the hypotheses of Theorem~\ref{thm:stima_L1_finitediff}. Then for every $p\in[1,+\infty]$ there exists a constant $C>0$ such that
\[
\sup_{n=0,\dots,N}\big\| u_n^\varepsilon(\cdot)-u(\cdot,n\Delta t)\big\|_{L^p(\T^d)}
\le C\,|\varepsilon|^{\frac12+\frac{1}{2p}}.
\]
\end{cor}

\begin{proof}
By the $L^\infty$ estimate \eqref{stima_Linfty_d}, and the $L^1$ estimate \eqref{stima_L1}, for each fixed $n$ we have, by relying on the interpolation inequality,
\[
\| u_n^\varepsilon(\cdot)-u(\cdot,n\Delta t)\|_{L^p(\T^d)}
\le \| u_n^\varepsilon(\cdot)-u(\cdot,n\Delta t)\|_{L^1(\T^d)}^{1/p}\,\| u_n^\varepsilon(\cdot)-u(\cdot,n\Delta t)\|_{L^\infty(\T^d)}^{1-1/p}
\]
\[\le C\,|\varepsilon|^{1/p}\,|\varepsilon|^{\frac12(1-1/p)}
= C\,|\varepsilon|^{\frac12+\frac{1}{2p}}.
\]
Taking the supremum over $n=0,\dots,N$ yields the bound.
\end{proof}

\subsection{The fully discrete finite difference scheme}\label{sec:fully_fd}
We now consider the fully discrete scheme. We fix an integer $I\ge 1$ and set $h:=1/I$. The Cartesian periodic grid on $\mathbb{T}^d$ is
\[
\mathcal{G}=\{x_i=hi:\ i=(i_1,\dots,i_d)\in\{0,\dots,I-1\}^d\}\subset\mathbb{T}^d,
\]
so that there are $I^d$ distinct grid points. All index operations are understood modulo $I$, which implements periodicity.
 Restricting the scheme \eqref{HJ_app_d} to the grid points yields the fully discrete scheme
\begin{equation}\label{HJ_app_fully_d}
\left\{
\begin{array}{ll}
\tilde u^\varepsilon_{n+1}(x_i)-\tilde u^\varepsilon_n(x_i)  + \Delta t \, F\!\big( - \delta_{h} \tilde u^\varepsilon_n(x_i)  ,\ \delta_{-h} \tilde u^\varepsilon_n(x_i) \big) =0, 
& \text{in} \quad \mathcal{G}\times\{0,\dots,N-1\},   \\[6pt]
\tilde u^\varepsilon_0(x_i)= g(x_i), & \text{on}\quad\mathcal{G}.
\end{array}\right.
\end{equation}
Let $\mathcal{I}[\tilde u^\varepsilon_n](\cdot)$ be the piecewise-multilinear interpolant of the grid values $u_n(x_i)$. 
For each multi-index $i=(i_1,\dots,i_d)$ denote the corresponding cell by
\(
Q_i = x_i + [0,h)^d.
\)
Then, for every $x\in Q_i$, the interpolated value is a convex combination of the $2^d$ corner values, namely
\begin{equation}\label{FD_sol}
\tilde u_n^\varepsilon(x):=\mathcal{I}[\tilde u^\varepsilon_n](x)=\sum_{\sigma\in\{0,1\}^d}\lambda_\sigma(x)\,\tilde u^\varepsilon_n(x_{i+\sigma}),
\qquad \lambda_\sigma(x)\ge0,\ \sum_{\sigma}\lambda_\sigma(x)=1.
\end{equation}
\begin{prop}\label{prop:stima_semi_fully}
Assume the hypotheses of Theorem~\ref{thm:stima_L1_finitediff}. Let $\tilde u_n^\varepsilon$ as in \eqref{FD_sol}. Then
	\begin{equation}\label{semidis_fullydis_d}
\sup_{n=0,\dots,N}\big\|u_n^\varepsilon(\cdot)-\tilde u_n^\varepsilon(\cdot)\big\|_{L^\infty(\T^d)} \le C\,h.
\end{equation}
\end{prop}

\begin{proof}
Fix $n$ and a cell $Q_{i}$. For $x\in Q_{ i}$, using $\sum_\sigma\lambda_\sigma=1$ and $\lambda_\sigma\ge0$, and that the fully discrete iterates coincide with the restriction of the semidiscrete solution to $\mathcal G$, we obtain
\[
\big|u_n^\varepsilon(x)-\tilde u_n^\varepsilon(x)\big|
\le \sum_{\sigma\in\{0,1\}^d}\lambda_\sigma(x)\big|u_n^\varepsilon(x)-\tilde u_n^\varepsilon(x_{i+\sigma})\big|
\]
\[
=\sum_{\sigma\in\{0,1\}^d}\lambda_\sigma(x)\big|u_n^\varepsilon(x)-u_n^\varepsilon(x_{i+\sigma})\big|.
\]
The semidiscrete solution $u_n^\varepsilon(\cdot)$ is Lipschitz in space uniformly in $\varepsilon$; denote a uniform Lipschitz constant by $R$. Hence the previous right-hand side is bounded by
\(
 R\max_{\sigma\in\{0,1\}^d}|x-x_{i+\sigma}|.
\)
For $x\in Q_{ i}$ every corner $x_{ i+\sigma}$ satisfies $|x-x_{ i+\sigma}|\le \sqrt{d}\,h$, therefore
\[
\big|u_n^\varepsilon(x)-\tilde u_n^\varepsilon(x)\big|\le R\sqrt{d}\,h.
\]
This proves \eqref{semidis_fullydis_d} for each $n$ with $C:=R\sqrt{d}$.
\end{proof}
\begin{rem}
While the estimate \eqref{stima_Linfty_d} addresses the semi-discrete scheme \eqref{HJ_app_d}, an analogous convergence rate holds for fully-discrete approximations \cite{cl}.
\end{rem}
By Theorem \ref{thm:stima_L1_finitediff}, Cor. \ref{cor:stima_Lp} and Prop. \ref{prop:stima_semi_fully}, we immediately get the error estimate for the fully discrete scheme
\begin{cor}
Assume the hypotheses of Theorem~\ref{thm:stima_L1_finitediff}. Then
	\begin{equation}\label{stima_ fully_L1}
		\sup_{n=0,\dots,N} \|\tilde u_n^\varepsilon(\cdot)-u(\cdot,n\dt)\|_{L^1(\T^d)} \le C\,|\varepsilon|.
	\end{equation}
Moreover, for any $p\in [1,+\infty]$
\[
\sup_{n=0,\dots,N}\|\tilde u_n^\varepsilon(\cdot)-u(\cdot,n\dt)\|_{L^p(\T^d)}\leq C|\varepsilon|^{\frac12+\frac{1}{2p}}.
\]
\end{cor}
\begin{rem}
The result of this section  can easily carry over to approximation scheme for the following
HJ equation:
\[
	\partial_t u + H(Du) = V(x,t) \qquad \text{in } Q := \mathbb{T}^d\times(0,T)
\]
with $V$ continuous in $\T^d\times (0,T)$ and uniformly Lipschitz continous in $x$.
\end{rem}

\section{The estimates for the semi-Lagrangian scheme}\label{sec:semi_lagrangian}
In this section, we derive an error estimate for a semi-Lagrangian approximation of the HJ equation (see \cite{csl,cdi,ff,g}). Following the strategy of the previous section, we first obtain the estimate for the semi-discrete scheme and then, using this result, we establish the estimate for the fully discrete scheme.\\
Since in this section we will exploit the control theoretic interpretation of the HJ equation to define the semi-Lagrangian scheme, we write the problem in the standard backward form
\begin{equation}\label{HJb}
	\begin{cases}
		-\partial_t u + H(x,Du) = 0, & \text{in }Q,\\[4pt]
		u(x,T)=g(x), & \text{on }\mathbb{T}^d.
	\end{cases}
\end{equation}
with 
\begin{equation}\label{Ham_separata}
	H(x,p)=H_0(p)-V(x)
\end{equation}
We assume that $H_0$ is superlinear and satisfies \textup{(H1)--(H2)} , $g$ satisfies \textup{(G1)}, and that $V$ satisfies
\begin{itemize}
	\item[(V1)] $V:\T^d\to\R$ is Lipschitz continuous and semiconcave.
\end{itemize}
We recall that the viscosity solution of \eqref{HJb} is given by the value function of a control problem. 
For a control law $\alpha:[t,T]\to\R^d$, consider the controlled dynamics
\begin{equation}\label{dyn} 
	\dot X(s)=-\alpha(s) \ \ \mbox{for } s\in (t,T), \ \  X(t)=x,
\end{equation}
and denote with $X^{x,t}[\alpha]$  the unique solution of the previous problem. Define the  cost function
\begin{equation}\label{cost}
	J(\alpha;x,t):=\int_t^T [ L( \alpha(s))+V(X^{x,t}[\alpha](s))]ds +g(X^{x,t}[\alpha](T)).
\end{equation}
where $L(\alpha)$ is the Legendre transform of $H_{0}(p)$. 
Then, the solution   $u:\T^d\times [0,T]\to \R $ of \eqref{HJb} is given by the value function  of  control problem with dynamics \eqref{dyn}  and cost \eqref{cost}, i.e.
\begin{equation}\label{value}
	u(x,t)=\inf_{\alpha \in \mathcal{A}(t)} J(\alpha;x,t) \qquad (x,t) \in  \T^d\times [0,T],
\end{equation}
where $\mathcal{A}(t)=L^2([t,T];\R^d)$ is the set of the control law. \\
The semi-discrete scheme is obtained by discretizing the previous control problem with respect to the time. Fix $\dt>0$, set $N=T/\dt$ (we assume $N$ to be an integer) and  for $n=0,...,N-1$ define the set of   discrete controls $\mA^\dt(n) :=\R^{d\times (N-n)}$. Given $\alpha=(\alpha_n,\alpha_{n+1},\dots,\alpha_{N-1})\in\mA^\dt(n)$, let us define the discrete dynamics $X_{k}^{x,n}[\alpha]$ as the solution of
\begin{equation}\label{dynh}
\left\{
\begin{array}{lcl}
	X_{k+1}=X_k-\dt\alpha_k = x- \dt \sum_{i=n}^{k}\alpha_{i}  \quad& \mbox{for } k=n,\dots,N-1,\\
	X_n=x,
\end{array}
\right.
\end{equation}
(when the context is clear, we simply write $X_k$). The  discrete cost function  $J^\dt( \cdot; x, n): \mA^\dt(n) \to \R$ is defined by
\begin{equation*}
J^\dt(\alpha;x,n)=\sum_{k=n}^{N-1}\dt [ L( \alpha_k)+V(X_k) ]  +g(X_N).
\end{equation*}
The time-discrete value function is
\begin{equation}\label{valueh}
   u^\dt_n(x)=\inf_{\alpha \in \mA^\dt(n)} J^\dt(\alpha ;x,n).
\end{equation}
Then, $\{u_n^\dt\}_{n=0}^{N}$ satisfies the   dynamic programming equation
\begin{equation}\label{HJh}
\left\{
\begin{split} 
&	u^\dt_n(x)=\inf_{\alpha\in \R^{d}}\left\{u^\dt_{n+1}(x-\alpha\dt)+ \dt (L(  \alpha)+  V(x))\right\}  \quad   x\in\T^d,\,  n=0,\dots,N-1,\\[6pt]
&	u^\dt_N(x)= g(x)    \hskip 24pt x\in\T^d.
\end{split}
\right.
\end{equation}
The next lemma provides useful properties of the discrete value function (see \cite[Lemma 3.2]{csl} and \cite{g}).
\begin{lemma}\label{lemma:prop_semilag} 
	Let $\{u^\dt_n\}_{n=0}^N$ be the solution of \eqref{HJh}. Then, the following assertions hold true: 
	\begin{itemize}
\item[(i)]  (Uniform bound) There exists a constant $C>0$ independent of $\dt$, such that for all \(n=0,\dots,N\) and \(x\in\T^d\),
  \[
    |u^\dt_n(x)| \le C(1+T).
  \]
\item[(ii)]  (Lipschitz in space and time) For all \(x,y\in\T^d\) and \(n,m\in\{0,\dots,N\}\),
  \begin{equation}\label{estim_Lip_corrected}
    |u^\dt_n(x)-u^\dt_m(y)| \le C_1\big(|x-y| + \Delta t\,|n-m|\big).
  \end{equation}
 \item[(iii)]  (Semiconcavity) For each \(n\in\{0,\dots,N\}\) the function \(\{u^\dt_n\}_{n=0}^N\) is semiconcave with a constant \(C_2\) independent of \(\Delta t\).
	\end{itemize}
\end{lemma}
The solution of the scheme \eqref{HJh} converges to the viscosity solution of \eqref{HJb}  with error estimates of order  $\half$ (see \cite{cdi,dfn}).
\begin{thm} 
	Let $\{u^\dt_n\}_{n=0}^N$ be the solution of \eqref{HJh} and   $u$   the viscosity solution of \eqref{HJb}. Then
	\begin{equation}\label{stima_Linfty_semilag}
		\sup_{n=0,\dots, N-1}\left\|u^\dt_n(\cdot) - u(\cdot,n\Delta t)\right\|_{L^\infty(\T^d)}\le C\dt^\half.
	\end{equation}
\end{thm}

It is clear that for every $(x,n) \in \T^{d}\times \{0,...,N-1\}$ there exists a least one optimal control law $\alpha\in \mA^\dt(n) $ such that
\begin{equation}	\label{optimal_control_law}           
	  u^\dt_n(x)=J^{\Delta t}(\alpha;x,n).
\end{equation} 
We denote by $\mA^{\dt,*}(x,n)$   the set of discrete controls $\alpha\in \mA^{\dt}(n)$ such that \eqref{optimal_control_law}  holds. The next lemma gives a crucial information about the relation between uniqueness of the optimal control law and the differentiability 
of $u^\dt_n$ at $x$ (for the proof we refer again to \cite[Lemma 3.3]{csl}).
\begin{lemma}\label{lemma:optimal_control}
Let $(x,n) \in \T^{d}\times \{0,...,N-1\}$ and $\alpha \in \mA^{\dt,*}(x,n)$ such that \eqref{optimal_control_law} holds. 
 Then, 
 \begin{itemize}
 	\item[(i)]
 		  The set $\mA^{\dt,*}(x,n) $  is bounded by a constant independent of $\dt$.
 \item[(ii)] For any  $\alpha \in \mA^{\dt,*}(x,n)$ and $k=n+1,..., N-1$, we have   $ \mA^{\dt,*}(X_{k}^{x,n}[\alpha],k)= \left\{ (\alpha_{k},..., \alpha_{N-1})\right\}.$
\item[(iii)] The function $u^\dt_n(\cdot)$ is differentiable at $x$  iff there exists $ {\alpha}(x,n) \in \mA^{\dt}(n)$ such that  $\mA^{\dt,*}(x,n)= \{ {\alpha}(x,n)\}$. In that case,
 	\begin{equation}\label{optimal_control_1}
 		D u^\dt_n(x)= D_\alpha L( {\alpha}_{n}(x,n))+ \dt DV(x).
 	\end{equation}
 	In particular, the problem associated with $u^\dt_n(x,n)$ has a unique solution a.e. in $\T^{d}$. 
\item[(iv)] For all $\alpha \in \mA^{\dt,*}(x,n)$,  the following relation hold true:
 	\begin{equation}\label{optimal_control_2}
 		Du^\dt_{k+1}(X_{k}-\dt\alpha_{k}(X_{k},k)) =D_\alpha L(\alpha_{k}(X_{k},k))   \quad \mbox{for } \, k=n,..., N-1.
 	\end{equation} 
 \end{itemize}

\end{lemma}
\begin{rem}
	We observe that, for $k=n$, \eqref{optimal_control_2} reads as
	\begin{equation}\label{optimal_control_3}
		Du^\dt_{n+1}(x-\dt\alpha_{n}(x,n)) =D_\alpha L(\alpha_{n}(x,n))   \quad \mbox{for } \, k=n,..., N-1.
	\end{equation}
\end{rem}

We  prove the $L^1$ estimate for the semi-Lagrangian scheme. The idea of the proof is similar to the one of Theorem \ref{thm:stima_L1_finitediff}.	 
\begin{thm}[$L^1$ error estimate; semi-Lagrangian scheme]\label{thm:stima_L1_semilag}
	Let $u$ be the solution of \eqref{HJb} and $\{u^\dt_n\}_{n=0}^N$ the  solution of \eqref{HJh}. Under \textup{(H1)-(H2), (V1), (G1)},  there exists \(C>0\) depending on \(T,H,g\) such that
	\begin{equation}\label{stima_L1_semilag}
		\sup_{n=0,\dots,N}\big\|u_n^\dt(\cdot)-u(\cdot,n\Delta t)\big\|_{L^1(\mathbb{T}^d)} \le C\,\dt.
	\end{equation}
\end{thm}

\begin{proof}
	We rewrite the first equation in \eqref{HJh} as
	\[
		\frac{u^\dt_{n}(x)-u^\dt_{n+1}(x)}{\dt}+F[u^\dt_{n+1}](x)-V(x)=0
	\]		
where 
	\[
	F[u^\dt_{n+1}](x)=	\sup_{\alpha\in \R^d}\left\{-\frac{u^\dt_{n+1}(x-\dt \alpha)-u^\dt_{n+1}(x)}{\dt}- L(\alpha)\right\}.
	\]	
	For \(n=0,\dots,N-1\) and \(t\in[n\Delta t,(n+1)\Delta t]\), define
	\begin{equation}\label{tim_inter_semilagr}
		u^\dt(x,t)=\Big(1-\frac{t-n\Delta t}{\Delta t}\Big)u^\dt_n(x)
		+\frac{t-n\Delta t}{\Delta t}\,u^\dt_{n+1}(x).
	\end{equation}
	Since $\partial_t u^\dt(x,t)= (u^\dt_{n+1}(x)-u^\dt_{n}(x))/\dt$, the function $z(x,t)=u^\dt(x,t)-u(x,t)$    solves  the equation
	\begin{equation}\label{HJ_pert}
		\begin{split}
			&z_t - b(x,t)Dz=F[u^\dt_{n+1}]-H_0(D u^\dt(t))\qquad t\in [n\dt, (n+1)\dt] 
		\end{split}
	\end{equation}
	with 
	\[ b(x,t)=\int_0^1D_pH_{0}(sD u^\dt+(1-s)D u)ds.\]
	Let $\tau=\bar n \Delta t$ for $\bar n\in\{0, \dots,N-1\}$. We multiply   equation \eqref{HJ_pert} for  the solution $\rho$ of the forward conservative equation
	\begin{equation}\label{HJ_adjoint}
	\left\{	\begin{array}{ll}
			\partial_t \rho -\mathrm{div}(b(x,t)\rho)=0\quad &\text{in $\T^d\times (\tau,T)$}\\
			\rho(\cdot, \tau)=\mathrm{sgn}(z(\cdot, \tau))  &\text{on $\T^d$},
		\end{array}
		\right.		
	\end{equation}
 and integrate on $\mathbb{T}^d\times (\tau,T)$  to get by duality
	and $z(x,T)=0$
	\begin{equation}\label{stima_1_semi}
		\begin{split}
			&\int_{\T^d}\rho(x,\tau)\,z(x,\tau)\,dx
			= \sum_{n=\bar n}^{N-1}\int_{n\dt}^{(n+1)\dt}\int_{\T^d}\Big[H_{0}(D u^\dt(t))-F[u^\dt_{n+1}]\Big]\rho \,dxdt\\
			&=\sum_{n=\bar n}^{N-1}\int_{n\dt}^{(n+1)\dt}\int_{\T^d}\Big[H_{0}(D u^\dt(t))-H_{0}(Du^\dt_{n+1})\Big]\rho dxdt\\
			&+\sum_{n=\bar n}^{N-1}\int_{n\dt}^{(n+1)\dt}\int_{\T^d}\Big[H_{0}(Du_{n+1}^\dt)-F[u^\dt_{n+1}]\Big]\rho dxdt
		\end{split}
	\end{equation}
	We estimate the first term  on the left hand side of \eqref{stima_1_semi}. We have
	\begin{equation}\label{stima_1_a}
		\begin{split}
			&	\int_{\T^d}\Big[H_{0}(D u^\dt(t))-H_{0}(Du^\dt_{n+1})\Big]\rho dx\\
			&	=\int_{\T^d}\left(\int_0^1 D_pH_{0}(sD u^\dt+(1-s)D u^\dt_{n+1})ds\right)(D u^\dt(t)- D u^\dt_{n+1})\rho dx\\
			&\le M\|\rho\|_{L^\infty(\T^d\times[\tau,T])}\int_{\T^d}|D u^\dt(t)-D u^\dt_{n+1}|\,dx.
		\end{split}
	\end{equation}
We have
\begin{equation}\label{stima_1a_semilag}
Du^\dt(x,t)- Du^\dt_{\,n+1}(x)
= \bigg(\frac{t-n\dt}{\dt}-1\bigg)D\big(u^\dt_{n+1}(x)-u^\dt_{\,n}(x)\big)
\end{equation}	
By	  \eqref{optimal_control_1} and	  \eqref{optimal_control_3}, we get
\[
D  u^\dt_n(x)=Du^\dt_{n+1}(x-\dt\alpha_{n}(x,n))+ \dt DV(x)
\]
and, replacing in \eqref{stima_1a_semilag}, we get, for a.e. $x\in\T^d$,
\begin{equation}\label{stima_1c_semilag}
	D u^\dt(t)- D u^\dt_{n+1}
	= \bigg(\frac{t-n\dt}{\dt}-1\bigg)
	\Big[Du^\dt_{n+1}(x)- Du^\dt_{n+1}(x-\dt\,\alpha_{n}(x,n))-\dt\, DV(x)\Big].
\end{equation}
Integrating over $\T^d$ and using Lemma~\ref{lemma:optimal_control}.(i), Lemma~\ref{lemma:prop_semilag}.(iii) and
Lemma~\ref{lemma:bound_hessian}.(ii), we get
\begin{equation}\label{stima_1d_semilag}
	\int_{\T^d}\bigl|D u^\dt_{n+1}(x)-D u^\dt_{n+1}(x-\dt\,\alpha_n(x,n))\bigr|dx
	\le C_0\,\dt\,\|\alpha_n(\cdot,n)\|_{L^\infty}
	\int_{\T^d}|D^2u^\dt_{n+1}|\,dx\le C\dt .
\end{equation}
By \eqref{stima_1_a}, \eqref{stima_1c_semilag}, \eqref{stima_1d_semilag},
\textup{(V1)}, we get
\begin{equation}\label{stima_1_semilag}
	\Big|\sum_{n=\bar n}^{N-1}\int_{n\dt}^{(n+1)\dt}\int_{\T^d}
	\Big[H_{0}(D u^\dt(t))-H_{0}(Du^\dt_{n+1})\Big]\rho\, dxdt\Big|\le C\dt .
\end{equation}
We now consider the second term in \eqref{stima_1_semi} and we show that
\begin{equation}\label{stima_G}
	\int_{\T^d}\bigl|H_0(Du^\dt_{n+1}(x))-F[u^\dt_{n+1}](x)\bigr|\,dx\le C\dt .
\end{equation}
We first prove a pointwise upper bound. For a.e. $x\in\T^d$ let
$\bar\alpha(x):=D_pH_0(Du^\dt_{n+1}(x))$, so that
$H_{0}(Du^\dt_{n+1}(x))={\bar\alpha(x)}\cdot Du^\dt_{n+1}(x)-L({\bar\alpha(x)})$
and, by Lemma~\ref{lemma:prop_semilag}.(ii),
$\|\bar\alpha\|_{L^\infty}\le C$. By the semiconcavity of $u^\dt_{n+1}$,
\[
	u^\dt_{n+1}(x-\dt\bar\alpha(x))\le u^\dt_{n+1}(x)
	-\dt\,\bar\alpha(x)\cdot Du^\dt_{n+1}(x)+C_2\dt^2|\bar\alpha(x)|^2 ,
\]
so that, choosing $\alpha=\bar\alpha(x)$ in the supremum defining
$F[u^\dt_{n+1}](x)$,
\begin{equation}\label{stima_G_sopra}
	H_0(Du^\dt_{n+1}(x))-F[u^\dt_{n+1}](x)\le C_2|\bar\alpha(x)|^2\dt\le C\dt
	\qquad\text{for a.e. }x\in\T^d .
\end{equation}
For the bound from below, observe that, by \eqref{HJh}, the supremum defining
$F[u^\dt_{n+1}](x)$ is attained at $\alpha_n(x,n)$ for a.e. $x$, and by the
Legendre inequality
$H_0(Du^\dt_{n+1}(x))\ge \alpha_n(x,n)\cdot Du^\dt_{n+1}(x)-L(\alpha_n(x,n))$ we
get
\[
	H_0(Du^\dt_{n+1}(x))-F[u^\dt_{n+1}](x)\ge
	\frac{u^\dt_{n+1}(x-\dt\alpha_n(x,n))-u^\dt_{n+1}(x)}{\dt}
	+\alpha_n(x,n)\cdot Du^\dt_{n+1}(x) .
\]
By Lemma~\ref{lemma:optimal_control}.(iv), $u^\dt_{n+1}$ is differentiable at
$x-\dt\alpha_n(x,n)$ for a.e. $x\in\T^d$; the semiconcavity inequality at that
point gives
\[
	u^\dt_{n+1}(x)\le u^\dt_{n+1}(x-\dt\alpha_n(x,n))
	+\dt\, Du^\dt_{n+1}(x-\dt\alpha_n(x,n))\cdot\alpha_n(x,n)
	+C_2\dt^2|\alpha_n(x,n)|^2 ,
\]
and hence
\[
	H_0(Du^\dt_{n+1}(x))-F[u^\dt_{n+1}](x)\ge
	\alpha_n(x,n)\cdot\big[Du^\dt_{n+1}(x)-Du^\dt_{n+1}(x-\dt\alpha_n(x,n))\big]
	-C_2|\alpha_n(x,n)|^2\dt .
\]
Combining the last inequality with \eqref{stima_G_sopra},
Lemma~\ref{lemma:optimal_control}.(i) and \eqref{stima_1d_semilag}, we obtain
\eqref{stima_G}. Therefore,
\[
	\Big|\sum_{n=\bar n}^{N-1}\int_{n\dt}^{(n+1)\dt}\int_{\T^d}
	\Big[H_{0}(Du^\dt_{n+1})-F[u^\dt_{n+1}]\Big]\rho(x,t)\, dxdt\Big|\]
\begin{equation}\label{stima_2_semilag}
	\le \|\rho\|_{L^{\infty}}\, T \sup_{n}\int_{\T^d}
	\bigl|H_0(Du^\dt_{n+1})-F[u^\dt_{n+1}]\bigr|\,dx
	\le \bar C\dt .
\end{equation}
Replacing \eqref{stima_1_semilag} and \eqref{stima_2_semilag} in
\eqref{stima_1_semi}, we get \eqref{stima_L1_semilag}.

\end{proof}
An argument similar to Cor. \ref{cor:stima_Lp} gives the following $L^p$-estimate.
\begin{cor}\label{cor:stima_Lp_semilag}
Assume the hypotheses of Theorem~\ref{thm:stima_L1_semilag}. Then for every $p\in [1,+\infty]$ there exists a constant $C>0$ such that
\[
\sup_{n=0,\dots,N}\big\| u_n^\dt(\cdot)-u(\cdot,n\Delta t)\big\|_{L^p(\T^d)}
\le C\,\dt^{\frac12+\frac{1}{2p}}.
\]
\end{cor}
\subsection{The fully discrete semi-Lagrangian scheme}
We now consider the fully discrete scheme, and exploit the same Cartesian periodic grid on $\mathbb{T}^d$ used in the subsection ~\ref{sec:fully_fd}. Let $v:\mathcal{G}\to \mathbb{R}$ be a grid function. Let us denote by $\mathcal{I}[v](\cdot)$ its piecewise-multilinear interpolation. The associated numerical scheme is as follows. 
\begin{equation}\label{HJhk}
\left\{
\begin{array}{ll}
u^\varepsilon_n(x_i)=\inf_{\alpha\in \R^{d}}\left\{\mathcal{I}[u^\varepsilon_{n+1}](x_i-\dt\alpha)+ \dt (L( \alpha)+V(x_i))\right\} \quad & \text{in }\mathcal{G}\times\{0,\dots,N-1\},\\[6pt]
u^\varepsilon_N(x_i)= g(x_i) & \text{on }\mathcal{G}.
\end{array}
\right.
\end{equation}
Then, we define the solution on the whole $\mathbb{T}^d$ by $\tilde u_n^\varepsilon(\cdot) := \mathcal{I}[u^\varepsilon_n](\cdot)$.

We now provide an estimate of the distance between the solutions of \eqref{HJhk} and \eqref{HJh} (see \cite{cfm} for the proof).
\begin{prop}\label{prop:stima_semi_fully_semilag} Let $\tilde u^{\varepsilon}$ be the piecewise-multilinear interpolant of the solution of \eqref{HJhk}, i.e., $\tilde u^{\varepsilon}_n =\mI[u^\varepsilon_n]
	$, and assume the hypotheses of Theorem \ref{thm:stima_L1_semilag}. Then,
	\begin{equation*}
		\sup_{n=0,\dots,N}\|u_n^\dt(\cdot) - \tilde u_n^{\varepsilon}(\cdot)\|_{L^\infty(\T^d)} \le C\frac{h}{\dt},
	\end{equation*}
	where $C$ is a constant independent on $\varepsilon$.
\end{prop}
Combining Theorem \ref{thm:stima_L1_semilag}, Corollary \ref{cor:stima_Lp_semilag}, and Proposition \ref{prop:stima_semi_fully_semilag}, we obtain the following $L^p$ estimate for the fully discrete scheme

\begin{cor}
	Let $u$ be the solution of \eqref{HJb}, and let $\tilde u_n^\varepsilon$ be defined as in  Prop. \ref{prop:stima_semi_fully_semilag}. Then
	\begin{equation*}
		\sup_{n=0,\dots,N}\big\| \tilde u_n^\varepsilon(\cdot)-u(\cdot,n\Delta t)\big\|_{L^1(\T^d)} \le C\left(\dt+\frac{h}{\dt}\right).
	\end{equation*}
	Moreover, for any $p\in [1,+\infty]$,
	\[
	\sup_{n=0,\dots,N}\|\tilde u_n^\varepsilon(\cdot)-u(\cdot,n\Delta t)\|_{L^p(\T^d)}\leq C\left(\dt+\frac{h}{\dt}\right)^{\frac{1}{p}}\left(\dt^\frac 12+\frac{h}{\dt}\right)^{1-\frac{1}{p}}.
	\]
	
\end{cor}

\section*{Acknowledgments}
The authors thank Alessandro Goffi for his useful suggestions.
Fabio Camilli is a member of GNAMPA–INdAM and acknowledges support from the PRIN-PNRR 2022 project “\textit{Some mathematical approaches to climate change and its impacts}”.

\appendix

\section{Proofs of technical lemmas}
\begin{proof}[Proof of Prop. \ref{prop:properties_scheme}]	

\noindent\textbf{(i)}
Let
\[
w_n^{(i)}(x):=\delta_h^{(i)}u_n^\varepsilon(x),\qquad
R_h:=\|\delta_h g\|_{L^\infty(\T^d)}
=\max_{1\le i\le d}\|\delta_h^{(i)}g\|_{L^\infty(\T^d)} .
\]
We first prove the discrete Lipschitz estimate by induction on \(n\). The base case
\(n=0\) is immediate. Assume the claim holds at time \(n\). Fix a direction \(i\).
By the induction hypothesis,
\[
u_n^\varepsilon(x+h e_i)-u_n^\varepsilon(x)
=h\,w_n^{(i)}(x)\le hR_h .
\]
Therefore
\[
u_n^\varepsilon(x+h e_i)\le u_n^\varepsilon(x)+hR_h .
\]
Consider the discrete evolution operator
\[
(\mathcal G_h v)(y):=
v(y)-\Delta t\,F\big(-\delta_h v(y),\delta_{-h}v(y)\big),
\]
so that \(u_{n+1}^\varepsilon=\mathcal G_h u_n^\varepsilon\). By the monotonicity of
\(\mathcal G_h\), and since adding a constant leaves the finite differences unchanged, we get
\[
\begin{aligned}
u_{n+1}^\varepsilon(x+h e_i)
&=(\mathcal G_h u_n^\varepsilon)(x+h e_i)  \\
&=\mathcal G_h\big(u_n^\varepsilon(\cdot+h e_i)\big)(x) \\
&\le \mathcal G_h\big(u_n^\varepsilon+hR_h\big)(x) \\
&=(\mathcal G_h u_n^\varepsilon)(x)+hR_h \\
&=u_{n+1}^\varepsilon(x)+hR_h .
\end{aligned}
\]
Equivalently,
\[
\delta_h^{(i)}u_{n+1}^\varepsilon(x)\le R_h .
\]
Applying the same argument to the inequality
\[
u_n^\varepsilon(x)\le u_n^\varepsilon(x+h e_i)+hR_h
\]
gives
\[
\delta_h^{(i)}u_{n+1}^\varepsilon(x)\ge -R_h .
\]
Hence
\[
\max_{1\le i\le d}
\|\delta_h^{(i)}u_{n+1}^\varepsilon\|_{L^\infty(\T^d)}
\le R_h .
\]
By induction,
\[
\|\delta_h u_n^\varepsilon\|_{L^\infty(\T^d)}
\le \|\delta_h g\|_{L^\infty(\T^d)}
\qquad\text{for every }n=0,\ldots,N .
\]
The same estimate for \(\delta_{-h}u_n^\varepsilon\) follows by translation.
It remains to prove the continuous Lipschitz estimate. Set
\[
L_g:=\|Dg\|_{L^\infty(\T^d)} .
\]
For \(z\in\R^d\), let \(T_zv(x):=v(x+z)\). Since \(g\) is Lipschitz,
\[
T_zg(x)=g(x+z)\le g(x)+L_g|z|
\qquad\text{for all }x\in\T^d,\ z\in\R^d .
\]
We prove by induction that
\[
u_n^\varepsilon(x+z)\le u_n^\varepsilon(x)+L_g|z|
\qquad\text{for all }x\in\T^d,\ z\in\R^d .
\]
The case \(n=0\) is the previous inequality for \(g\). Assume it holds at time \(n\).
Using the translation invariance, monotonicity, and constant-preserving property of
\(\mathcal G_h\), we obtain
\[
\begin{aligned}
u_{n+1}^\varepsilon(x+z)
&=(\mathcal G_h u_n^\varepsilon)(x+z) \\
&=\mathcal G_h(T_z u_n^\varepsilon)(x) \\
&\le \mathcal G_h(u_n^\varepsilon+L_g|z|)(x) \\
&=(\mathcal G_h u_n^\varepsilon)(x)+L_g|z| \\
&=u_{n+1}^\varepsilon(x)+L_g|z| .
\end{aligned}
\]
Replacing \(z\) by \(-z\) gives
\[
u_{n+1}^\varepsilon(x)-u_{n+1}^\varepsilon(x+z)\le L_g|z| .
\]
Therefore
\[
|u_{n+1}^\varepsilon(x+z)-u_{n+1}^\varepsilon(x)|
\le L_g|z|
\qquad\text{for all }x,z .
\]
Thus \(u_{n+1}^\varepsilon\) is \(L_g\)-Lipschitz. By induction,
\[
\operatorname{Lip}(u_n^\varepsilon)\le \operatorname{Lip}(g)
=\|Dg\|_{L^\infty(\T^d)} .
\]
Since \(u_n^\varepsilon\) is Lipschitz, it is differentiable a.e. by Rademacher's theorem, and
\[
\|D u_n^\varepsilon\|_{L^\infty(\T^d)}
\le \|Dg\|_{L^\infty(\T^d)} .
\]

		\noindent\textbf{(ii)}
Let
\(
w_n(x):=u_n^\varepsilon(x)+n\Delta t\,H(0),
\)
so that \(\delta_{\pm h}w_n=\delta_{\pm h}u_n^\varepsilon\). By the discrete Lipschitz bound proved in part (1) there exists \(R>0\) such that for every \(n\) and every coordinate \(i\)
\(
\sup_{x\in\T^d}|\delta^{(i)}_{\pm h} w_n(x)| \le R.
\)

From the scheme \eqref{HJ_app_d} we have pointwise for every \(x\in\T^d\)
\[
w_{n+1}(x)-w_n(x)
= -\Delta t\Big(F\big(-\delta_h w_n(x),\delta_{-h} w_n(x)\big)-F(0,0)\Big),
\]

Apply the multivariate mean value theorem to \(F:\R^{2d}\to\R\) on the segment joining \((0,0)\) and \(\big(-\delta_h w_n(x),\delta_{-h} w_n(x)\big)\). There exists \(\xi(x)\in\R^{2d}\) on that segment such that
\[
\begin{aligned}
F\big(-\delta_h w_n(x),\delta_{-h} w_n(x)\big)-F(0,0)
&= DF(\xi(x))\cdot\big(-\delta_h w_n(x),\delta_{-h} w_n(x)\big)\\
\end{aligned}
\]

Substituting into the equation of interest, yields the identity
\[
w_{n+1}(x)-w_n(x)
= -\frac{\Delta t}{h}\sum_{i=1}^d\Big[ F_{p_i}(\xi(x))\big(-w_n(x+h e_i)+w_n(x)\big)+ F_{q_i}(\xi(x))\big(w_n(x)-w_n(x-h e_i)\big)\Big],
\]
Therefore,
\[
\begin{aligned}
w_{n+1}(x)
&= \frac{\Delta t}{h}\sum_{i=1}^d F_{p_i}(\xi(x))\,w_n(x+h e_i)
+ \Big(1-\frac{\Delta t}{h}\sum_{i=1}^d\big(F_{p_i}(\xi(x))+F_{q_i}(\xi(x))\big)\Big) w_n(x)\\
&\qquad + \frac{\Delta t}{h}\sum_{i=1}^d F_{q_i}(\xi(x))\,w_n(x-h e_i).
\end{aligned}
\]
Since \(F_{p_i},F_{q_i}\ge0\), and the CFL condition holds, the three coefficients are nonnegative and sum to \(1\). It follows that, for each fixed \(x\),
	$$
		w_{n+1}(x)\le \max\{w_n(x),w_n(x\pm he_1),...,w_n(x\pm he_d)\}\le \sup_y w_n(y),
	$$
	hence \(\sup_x w_{n+1}(x)\le \sup_x w_n(x)\). 
	Similarly \(w_{n+1}(x)\ge \min\{w_n(x),w_n(x\pm he_1),...,w_n(x\pm he_d)\}\ge \inf_y w_n(y)\). Iterating we get \(\|w_n\|_{L^\infty(\T^d)}\le \|w_0\|_{L^\infty(\T^d)}=\|g\|_{L^\infty(\T^d)}\), and therefore
\[
\|u_n^\varepsilon\|_{L^\infty(\T^d)}
\le \|w_n\|_{L^\infty(\T^d)} + n\Delta t\,|H(0)|
\le \|g\|_{L^\infty(\T^d)} + n\Delta t\,|H(0)|.
\]

	\medskip
	\noindent\textbf{(iii)}
	Let \(w_n(x):=\Delta_k u^\varepsilon_n(x)/|k|^2\), that is, the quantity we want to estimate \ref{semiconc_solutions_finite_diff}.
	Introducing
	\(
	A_n(y):=\big(-\delta_h u^\varepsilon_n(y),\; \delta_{-h}u^\varepsilon_n(y)\big)
	\), evaluating the scheme \eqref{HJ_app_d} at the points \(y=x-k,x,x+k\), and suitably combining the three relations, we get
	\begin{equation}\label{eq:sn_evol}
		w_{n+1}(x)-w_n(x)
		= -\frac{\Delta t}{|k|^2}\Big( F\big(A_n(x+k)\big)-2F\big(A_n(x)\big)+F\big(A_n(x-k)\big)\Big).
	\end{equation}
	
	Since \(F\in C^2(\mathbb R^{2d})\) is convex, we have for any vectors \(X,Y,Z\in\mathbb R^{2d}\) \(F(X)-F(Y)\ge D F(Y)\cdot (X-Y)\) and \(F(Z)-F(Y)\ge D F(Y)\cdot (Z-Y)\). Therefore,
	\[
	F(X)-2F(Y)+F(Z)\ge D F(Y)\cdot (X-2Y+Z).
	\]
	Applying this inequality with
	\(X=A_n(x+k),\,Y=A_n(x),\,Z=A_n(x-k)\) in \eqref{eq:sn_evol} gives the estimate
	\[
		w_{n+1}(x)-w_n(x)
		\le -\frac{\Delta t}{|k|^2}\,D F\big(A_n(x)\big)\cdot\big(A_n(x+k)-2A_n(x)+A_n(x-k)\big)\]
Using \(\Delta_k u_n^\varepsilon=|k|^2 w_n\), and \(
\Delta_k\big(\delta^{(i)}_h u_n^\varepsilon\big)(x)
= \delta^{(i)}_h\big(\Delta_k u_n^\varepsilon\big)(x)
\), we get
\[
\Delta_k\big(\delta^{(i)}_h u_n^\varepsilon\big)(x)
=|k|^2\delta^{(i)}_h w_n(x),\qquad
\Delta_k\big(\delta^{(i)}_{-h} u_n^\varepsilon\big)(x)=|k|^2\delta^{(i)}_{-h} w_n(x),
\]
so that
\[
A_n(x+k)-2A_n(x)+A_n(x-k)=\Delta_k A_n(x)=|k|^2\big(-\delta_h w_n,\;\delta_{-h} w_n\big)(x).
\]
Therefore,
\[
w_{n+1}(x)-w_n(x)
\le -\Delta t\sum_{i=1}^d\big(-F_{p_i}(A_n(x))\,\delta^{(i)}_h w_n(x)
+F_{q_i}(A_n(x))\,\delta^{(i)}_{-h} w_n(x)\big).
\]
Writing \(\delta^{(i)}_{\pm h}w_n\) explicitly gives
\[
w_{n+1}(x)
\le w_n(x) -\frac{\Delta t}{h}\sum_{i=1}^d\Big(-F_{p_i}w_n(x+h e_i)+(F_{p_i}+F_{q_i})w_n(x)-F_{q_i}w_n(x-h e_i)\Big)
\]
\[
=w_n(x)\bigg(1 -\frac{\Delta t}{h}\sum_{i=1}^d(F_{p_i}+F_{q_i})\bigg)+\frac{\Delta t}{h}\sum_{i=1}^dF_{p_i}w_n(x+h e_i) +\frac{\Delta t}{h}\sum_{i=1}^dF_{q_i}w_n(x-h e_i).
\]
Since the scheme satisfies the monotonicity and CFL conditions, then the coefficients in the previous inequality are nonnegative and sum to $1$. Hence,
\[
w_{n+1}(x)\le\max\{w_n(x),\,w_n(x\pm h e_i)\}\le\sup_{y} w_n(y).
\]
Taking the supremum in $x$ yields $\sup_x w_{n+1}(x)\le\sup_x w_n(x)$, therefore the sequence $\{\sup_x w_n\}_n$ is nonincreasing. Consequently
\[
\sup_{x\in\T^d,\;k\in\T^d\setminus\{0\}}\frac{\Delta_k u_n^\varepsilon(x)}{|k|^2}
\le \sup_{x\in\T^d,\;k\in\T^d\setminus\{0\}} w_0(x)
= \sup_{x\in\T^d,\;k\in\T^d\setminus\{0\}}\frac{\Delta_k g(x)}{|k|^2}
\le C_{\mathrm{conc}}.
\]
\end{proof}
\begin{proof}[Proof of Lemma \ref{lemma:finite_difference}]
	Fix a direction \(i\). We have, for every \(x\in\T^d\),
	\[
	\delta_{h}^{(i)} v(x) - D_i v(x)
	= \frac{1}{h}\int_0^h\big(D_i v(x+s e_i)-D_i v(x)\big)\,ds
	= \frac{1}{h}\int_0^h\int_0^s D_{ii} v(x+\tau e_i)\,d\tau\,ds.
	\]
	Changing variables $\tau=\theta h$, $s=\sigma h$ yields the representation
	\[
	\delta_{h}^{(i)} v(x) - D_i v(x)
	= h\int_0^1 (1-\theta)\,D_{ii} v(x+\theta h e_i)\,d\theta.
	\]
	Therefore the pointwise estimate
	\[
	\big|\delta_{h}^{(i)} v(x) - D_i v(x)\big|
	\le h\int_0^1 \big|D_{ii} v(x+\theta h e_i)\big|\,d\theta
	\]
	holds. Integrating over \(\T^d\),
	\[
	\int_{\T^d}\big|\delta_{h}^{(i)} v(x) - D_i v(x)\big|\,dx
	\le h\int_0^1\int_{\T^d}\big|D_{ii} v(x+\theta h e_i)\big|\,dx\,d\theta
	= h\int_{\T^d}\big|D_{ii} v(y)\big|\,dy.
	\]
	To bound $\int_{\T^d}|D_{ii} v|$, see Lemma \ref{lemma:bound_hessian} .
\end{proof}

	\end{document}